\documentclass[11pt]{article}

\usepackage[utf8]{inputenc}
\usepackage[T1]{fontenc}
\usepackage{lmodern}
\usepackage{geometry}
\usepackage{amssymb, amsfonts, amsthm, amsmath}
\usepackage[english]{babel}
\usepackage[colorlinks]{hyperref}
\usepackage{tikz}
\usepackage{bbm}

\theoremstyle{plain}
\newtheorem{theorem}{Theorem}[section]

\newtheorem{lemma}[theorem]{Lemma}
\newtheorem{proposition}[theorem]{Proposition}

\theoremstyle{definition}

\theoremstyle{remark}

\newcommand{\N}{\mathbb{N}}

\newcommand{\R}{\mathbb{R}}
\newcommand{\C}{\mathbb{C}}

\newcommand{\mn}{\mathbb{N}}

\newcommand{\ind}[1]{\1_{\left\{#1\right\}}}

\numberwithin{equation}{section}

\DeclareMathOperator{\E}{\mathbb{E}}

\renewcommand{\P}{\mathbb{P}}

\newcommand{\calF}{\mathcal{F}}

\newcommand{\dd}{\mathrm{d}}
\newcommand{\mr}{\mathbb{R}}

\renewcommand{\bar}[1]{\overline{#1}}

\renewcommand{\tilde}[1]{\widetilde{#1}}
\renewcommand{\hat}[1]{\widehat{#1}}

\renewcommand{\rho}{\varrho}
\renewcommand{\epsilon}{\varepsilon}

\DeclareMathOperator{\1}{\mathbbm{1}}

\newcommand{\x}{\mathbf{x}}

\title{A result on power moments of Lévy-type perpetuities and its application to the $L_p$-convergence of Biggins' martingales in
branching Lévy processes}
\author{Alexander Iksanov\footnote{Faculty of Computer Science and Cybernetics, Taras Shevchenko National University of Kyiv, Kyiv,
Ukraine; e-mail: iksan@univ.kiev.ua} \ and \ Bastien
Mallein\footnote{LAGA - Institut Galilée, Université
Paris, France; e-mail: mallein@math.univ-paris13}}
\date{\today}

\begin{document}

\maketitle

\begin{abstract}
Lévy-type perpetuities being the a.s.\@ limits of
particular generalized Ornstein-Uh\-len\-beck processes are a natural
continuous-time generalization of discrete-time perpetuities.
These are random variables of the form
$S:=\int_{[0,\infty)}e^{-X_{s-}}{\rm d}Z_s$, where $(X,Z)$ is a
two-dimensional Lévy process, and $Z$ is a drift-free Lévy
process of bounded variation. We prove an ultimate criterion for
the finiteness of power moments of $S$. This result and the
previously known assertion due to Erickson and Maller
\cite{Erickson+Maller:2005} concerning the a.s.\@ finiteness of $S$
are then used to derive ultimate necessary and sufficient
conditions for the $L_p$-convergence for $p>1$ and $p=1$,
respectively, of Biggins' martingales associated to branching
Lévy processes. In particular, we provide final versions of
results obtained recently by Bertoin and Mallein in
\cite{Bertoin+Mallein:2018b}.
\end{abstract}



\section{Introduction}

Let $(M_k, Q_k)_{k\in\mn}$ be a sequence of independent copies of
an $\R^2$-valued random vector $(M,Q)$ with arbitrary dependence
of components. Further, denote by $(\Pi_n)_{n\in\mn_0}$ the
multiplicative (ordinary) random walk with factors $M_n$ for
$n\in\mn$ which starts at $1$, that is, $\Pi_0:= 0$ and
$\Pi_n:=\prod_{i=1}^n M_i$, $n \in \mn$. Then define its perturbed
variant $(\Theta_n)_{n\in\mn}$, that may be called a {\it
perturbed multiplicative random walk}, by
\begin{equation} \label{eq:MPRW}
\Theta_n :=\Pi_{n-1}Q_n,    \quad   n \in \mn.
\end{equation}
When $M_k$ and $Q_k$ are a.s.\@ positive, the random sequence
$(\log \Theta_n)_{n\in\mn}$ is known in the literature as a {\it
perturbed (additive) random walk}. A major part of the recent book
\cite{Iksanov:2016} is concerned with the so defined perturbed
random walks, both multiplicative and additive. We refer to the
cited book for numerous applications of these random sequences
and to \cite{Alsmeyer+Buraczewski+Iksanov:2017,
Buraczewski+Dyszewski+Iksanov+Marynych:2018,
Damek+Kolodziejek:2018, Iksanov+Jedidi+Bouzzefour:2018,
Iksanov+Pilipenko+Samoilenko:2017} for more recent contributions.

Recall that, provided that the series $\sum_{k\ge 1}\Theta_k$
converges a.s., its sum
\begin{equation*}\label{per}
Z:= \sum_{k\ge 1}\Theta_k
\end{equation*}
is called {\it perpetuity}. The term stems from the fact that such
random series may be used in insurance mathematics and financial
mathematics to model sums of discounted payment streams. The state
of the art concerning various aspects of perpetuities is discussed
in \cite{Buraczewski+Damek+Mikosch:2016} and \cite{Iksanov:2016}.
We think that the most valuable feature of the
perturbed multiplicative random walks is their link with
perpetuities.

There is also an unexpected connection, unveiled
in \cite{Lyons:1997} and detailed in \cite{Iksanov:2004} and
\cite{Alsmeyer+Iksanov:2009}, between perpetuities and branching
random walks. The connection, which is not immediately seen, emerges
when studying the weighted random tree associated with the
branching random walk under a size-biased measure. In particular,
criteria for the uniform integrability and the $L_p$-convergence
for $p>1$ of the Biggins martingale (also known as the additive
martingale or the intrinsic martingale in the branching random
walk) are closely linked with criteria for the a.s.\@ finiteness
and the existence of the $p$th moment of perpetuities,
respectively. In this way one arrives at a final version of the
famous Biggins martingale convergence theorem which was originally
proved by Biggins himself in \cite{Biggins:1977} with the help of
a different argument and under additional moment assumptions. The
recent article \cite{Bertoin+Mallein:2018} is aimed at obtaining
sufficient conditions for the uniform integrability and the
$L_p$-convergence for $p\in (1,2]$ of the Biggins martingale in a
{\it branching Lévy process}. To this end, a connection
similar to that described at the beginning of the paragraph is
exploited between certain continuous-time perpetuities and
branching Lévy process. The conditions obtained in
\cite{Bertoin+Mallein:2018} are not optimal.

In this article we first define {\it perturbed multiplicative
Lévy processes} which are natural continuous-time counterparts
of the perturbed multiplicative random walks. These are then used
to construct {\it Lévy-type perpetuities} in the same way as
the perturbed multiplicative random walks are used to construct
the discrete-type perpetuities. The Lévy-type perpetuities are
a particular instance of the limit random variables for
generalized Ornstein-Uhlenbeck processes. This restriction (that
is, that we consider the particular rather than any limit) is
motivated by a prospective application, see the end of this
section for more details. Necessary and sufficient conditions for
the a.s.\@ finiteness of the Lévy-type perpetuities can be
derived from \cite[Theorem
2]{Erickson+Maller:2005}.

Our main contribution is two-fold. First, we prove an ultimate criterion for the finiteness of the
$p$th moment of the Lévy-type perpetuity for all $p>0$.
Second, we apply this criterion and the aforementioned result from
\cite{Erickson+Maller:2005} to derive necessary and sufficient
conditions for the a.s.\@ and the $L_p$-convergence for $p\geq 1$
of the Biggins martingale in the branching Lévy process. Thus,
we obtain final versions of Theorem 1.1 and Proposition 1.4 in
\cite{Bertoin+Mallein:2018} which was our primary motivation.

\section{Lévy-type perpetuities}\label{defi}

In this section we first define a continuous-time counterpart of
the perturbed multiplicative random walks, described in \eqref{eq:MPRW}.

Let $\Lambda$ be a sigma-finite measure on $\R \times \R$ with
$\Lambda(\{0,0\}) = 0$. Define the projections $\Lambda_1$ and
$\Lambda_2$ of $\Lambda$ by
\[\Lambda_1(B):=\int_\mr \Lambda(B,{\rm d}y)\quad\text{and}\quad \Lambda_2(B):=\int_\mr \Lambda({\rm d}x, B)\]
for Borel sets $B$ in $\mr\backslash\{0\}$. Throughout the article
our standing assumption is that
\begin{equation}\label{assump}
\int_\R (x^2 \wedge 1) \Lambda_1({\rm d}x) < \infty \quad
\text{and} \quad \int_\R (|y|\wedge 1) \Lambda_2({\rm d}y) <
\infty.
\end{equation}
Denote by $N :=\sum_k \varepsilon_{(\tau_k, (i_k, j_k))}$ a
Poisson random measure on $\mr_+\times \mr^2$ with mean measure
${\rm LEB}\otimes \Lambda$, where $\mr_+:=[0,\infty)$,
$\varepsilon_{(t,(x,y))}$ denotes the Dirac mass
at $(t,(x,y))\subset\mr_+\times \mr^2$, and ${\rm LEB}$ is the
Lebesgue measure on $\mr_+$. Define $N_1:=\sum_k
\varepsilon_{(\tau_k, i_k)}$ and $N_2:=\sum_k
\varepsilon_{(\tau_k, j_k)}$, the projections of $N$. These are
Poisson random measures on $\mr_+\times \mr$ with mean measures
${\rm LEB}\otimes \Lambda_j$, $j=1,2$.

For $t\geq 0$, set
\begin{align}\label{levy}
X_t &:=vB_t+bt+ \int_{[0,\,t] \times \R} x\1_{[-1,1]}(x) N_1^c(\dd s \dd x) + \int_{[0,\,t] \times\R} x\1_{\R\backslash [-1,1]}(x) N_1(\dd s \dd x)\\
Z_t &:= \int_{[0,\,t]\times \R} y N_2(\dd s \dd y)\notag
\end{align}
where $v^2\geq 0$, $b\in\mr$ and $(B_t)_{t\geq 0}$ is a Brownian
motion independent of $N$. The first integral in \eqref{levy} is a
compensated Poisson integral (hence, the notation
$N_1^c$) which can be defined as the following
limit in $L_2$
\[\lim_{\delta\downarrow0} \int_{[0,\,t] \times \R} x\1_{(\delta, 1]}(|x|)
N_1(\dd s \dd x) - t\int_{\delta<|x|\leq
1}x\Lambda_1({\rm d}x).\]

In view of the second assumption in
\eqref{assump} the process $Z:=(Z_t)_{t\geq 0}$ is a drift-free
Lévy process of bounded variation. In particular, $Z$ can be
represented as the difference of two independent subordinators.
The random measure $N$ is the measure of jumps of the
two-dimensional Lévy process $(X_t, Z_t)_{t\geq 0}$.

Define the random process $Y:=(Y_t)_{t\geq 0}$ by
\[Y_t = \begin{cases}
    y, &\text{ if } N_2(\{t\}\times\{y\}) = 1;\\
    0, &\text{ if } N_2(\{t\}\times \R) = 0,
  \end{cases}
\]
that is, $Y=(Z_t-Z_{t-})_{t\geq 0}$ is the process of jumps of $Z$.
The process $(Y_te^{-X_{t-}})_{t\geq 0}$ which is a natural
continuous-time generalization of the discrete-time process
$(\Theta_n)_{n\in\mn}$ defined in~\eqref{eq:MPRW} will be called
{\it perturbed multiplicative Lévy process}. For $t\geq 0$,
set
\begin{equation}\label{st}
S_t: = \sum_{0\leq s\leq t}e^{-X_{s-}}Y_s=\sum_{\tau_k\leq t}
e^{-X_{\tau_k-}}j_k=\int_{[0,\,t]}e^{-X_{s-}}{\rm d}Z_s.
\end{equation}
Whenever the a.s.\@ limit $S:=\lim_{t\to\infty}S_t$ exists and is
finite, we call the random variable
\begin{equation}\label{perp}
S = \sum_{s \geq 0} e^{-X_{s-}}Y_s=\sum_k
e^{-X_{\tau_k-}}j_k=\int_{\mr_+}e^{-X_{s-}}{\rm d}Z_s
\end{equation}
{\it Lévy-type perpetuity}.

The following result which gives necessary and sufficient
conditions for the a.s.\@ finiteness of Lévy-type perpetuities
is a specialization\footnote{In the cited result $Z$ is allowed to
be an arbitrary Lévy process. The random process $(S_t)_{t\geq
0}$ in \eqref{st} is then called a generalized Ornstein-Uhlenbeck
process. In view of the second condition in \eqref{assump} which
is motivated by a forthcoming application of our results to
branching Lévy processes we only consider a subclass of
generalized Ornstein-Uhlenbeck processes.} of Theorem 2 in
\cite{Erickson+Maller:2005}. For $x\geq 1$, set
\[A(x):=1+\int_1^x
\Lambda_1((y,\infty)){\rm d}y=1+\int_\R (x \wedge z-1)_+
\Lambda_1(\dd z),\]
where $z_+ = \max(z,0)$ and
$y \wedge z = \min(y,z)$ for all $y,z \in \R$.
\begin{proposition}\label{maller}
Assume that
\begin{equation}\label{ex}
\lim_{t\to\infty}X_t=+\infty\quad\text{{\rm
a.s.}}\quad\text{and}\quad \int_{\R\backslash[-e, e]}\frac{\log
|y|}{A(\log |y|)}\Lambda_2({\rm d}y)<\infty.
\end{equation}
Then
\begin{equation}\label{exist}
\P\{\lim_{t\to\infty}S_t~\text{exists and is finite}\}=1.
\end{equation}
Conversely, if \eqref{ex} fails, then \eqref{exist} fails.
\end{proposition}
It should not come as a surprise that Proposition \ref{maller} is
very similar to Theorem 2.1 in \cite{Goldie+Maller:2000} which
provides a criterion for the a.s.\@ finiteness of discrete-time
perpetuities $Z$.

\section{Power moments of Lévy-type perpetuities}

\subsection{Main result}\label{ma}

The purpose of this section is to point out
necessary and sufficient conditions for the finiteness of power
moments of $S$. Before formulating the
corresponding result we note that the distribution of $S$ is
degenerate if, and only if, it is degenerate at $0$, and that the
latter occurs if, and only if, $\Lambda_2$ is trivial which means
that $\Lambda_2\equiv 0$. The non-obvious part of this statement,
that is, that the distribution of $S$ cannot be degenerate at a
nonzero point follows from the fact that $Z$ does not have a
Brownian component and Theorem 2.2 in
\cite{Bertoin+Lindner+Maller:2008}.

\begin{theorem}\label{moments}
Assume that $\Lambda_2$ is nontrivial and
let $p>0$. The following assertions are equivalent:
\begin{equation}\label{mom11}
\E e^{-pX_1}<1\quad\text{and}\quad \int_{\R\backslash [-1,1]}|y|^p
\Lambda_2({\rm d}y)<\infty;
\end{equation}
\begin{equation}\label{mom12}
\E |S|^p<\infty.
\end{equation}
\end{theorem}

\subsection{Auxiliary results}

Proposition \ref{air} and Proposition \ref{behme}
given below are our main technical tools for the proof of Theorem
\ref{moments}. We start by recalling a criterion obtained in
Theorem 1.4 of \cite{Alsmeyer+Iksanov+Roesler:2009} for the
finiteness of power
moments of discrete-time perpetuities $Z$.
\begin{proposition}\label{air}
Let $p>0$ and suppose that
\begin{equation}\label{nondeg}
\P\{M=0\}=0\quad\text{and}\quad \P\{Q=0\}<1
\end{equation}
and that
\begin{equation}\label{nondeg2}
\P\{Q+Mr=r\}<1\quad \text{for all}~r\in\R.
\end{equation}
The following assertions are equivalent:
\begin{equation}\label{mom1}
\E |M|^p<1\quad\text{and}\quad \E |Q|^p<\infty;
\end{equation}
\begin{equation}\label{mom2}
\E |Z|^p<\infty.
\end{equation}
\end{proposition}

The next proposition gives sufficient conditions for the finiteness of the $p$th moment of the integral of an adapted process against the Lévy process $Z$ defined in Section \ref{defi}.

\begin{proposition}\label{behme}
Let $(Z_s)_{s\geq 0}$ be a drift-free Lévy process of finite
variation (as defined in Section \ref{defi})  and $(H_s)_{s\geq
0}$ an adapted c\`{a}dl\`{a}g process. Suppose that there exists
$p>0$ such that $\E |Z_1|^p<\infty$ and $\E\sup_{s\in
(0,1]}|H_s|^p<\infty$. Then
\[\E \Big|\int_{(0,1]} H_{s-}{\rm d}Z_s\Big|^p<\infty.\]
\end{proposition}
\begin{proof}
When $p\geq 1$ the assertion follows from Lemma 6.1 in
\cite{Behme:2011}.

Assume that $p \in (0,1)$. 
Subadditivity of $x\mapsto x^p$ on $\R_+$ and the triangle inequality entail
\begin{equation}\label{above}
\E \Big|\int_{(0,1]} H_{s-}{\rm d}Z_s\Big|^p \leq \E\left(
\int_{(0,1]} \left|H_{s-}\right| {\rm d} Z^{(1)}_s \right)^p + \E\left(
\int_{(0,1]} \left|H_{s-}\right| {\rm d} Z^{(2)}_s  \right)^p,
\end{equation}
where, for $t\geq 0$,
\[Z^{(1)}_t: = \int_{[0,\,t] \times \R}|y|\1_{[-1,1]}(y) N_2(\dd s \dd
y)=\sum_{\tau_k\leq t} |j_k| \1_{\{|j_k|\leq 1\}}
\]
and
\[Z^{(2)}_t: =\int_{[0,\,t] \times \R}
|y|\1_{\R\backslash [-1,1]}(y) N_2(\dd s \dd y)=\sum_{\tau_k\leq
t}|j_k|\1_{\{|j_k|>1\}}.\]
Note that $Z^{(i)}:=(Z_t^{(i)})_{t\geq 0}$, $i=1,2$ are drift-free subordinators. We shall prove finiteness of the two summands on the right-hand side of \eqref{above} separately.

We start by observing that $Z^{(2)}$ is a compound Poisson process with jumps sizes larger than one.
Denote by $T_1, T_2,\ldots$ the times at which $Z^{(2)}$ jumps, ranked in the increasing order, and set $R_i:=Z^{(2)}_{T_i}-Z^{(2)}_{T_i-}$ for $i\in\mn$. The sequence $(T_k)_{k\in\mn}$ forms the arrival times of a Poisson process with intensity $c:=\Lambda_2(\mr\backslash [-1, 1])$, and $(R_k)_{k\in\mn}$ are i.i.d.\@ random variables with distribution $\P\{R_1>x\}=c^{-1}\Lambda_2(\R\backslash[-x,x])$ for $x>1$ and $\P\{R_1>x\}=1$ for $x\leq 1$.
Moreover, for each fixed $i\in\mn$, $(H_{T_i-}, T_i)$ is independent of $R_i$. Using these facts
in combination with the aforementioned subadditivity we obtain
\begin{eqnarray*}
\E \Big|\int_{(0,\,1]}H_{s-}{\rm d}Z^{(2)}_s\Big|^p &\leq& \E\Big( \sum_{i\geq 1} |H_{T_i-}|^p R_i^p\1_{\{T_i\leq 1\}} \Big)
=\E\Big(\sum_{i\geq 1} |H_{T_i-}|^p \1_{\{T_i\leq 1\}} \Big) \E R_1^p\\
&=& c \E \Big(\int_0^1 |H_s|^p {\rm d}s\Big)  c^{-1}\int_{\mr\backslash[-1,\,1]} |y|^p \Lambda_2({\rm d} y),
\end{eqnarray*}
where, recalling that $(H_s)_{s\geq 0}$ is an adapted process, the second equality is justified by the compensation formula for Poisson random measures. As a result,
\[
\E \Big|\int_{(0,\,1]}H_{s-}{\rm d}Z^{(2)}_s\Big|^p \leq \E \Big(\sup_{s \in [0,1]} |H_s|^p \Big)\int_{\mr\backslash[-1,\,1]} |y|^p \Lambda_2({\rm d} y) < \infty.
\]

It remains to show that
\begin{equation}\label{fin}
\E\left(\int_{(0,1]} \left|H_{s-} \right| {\rm d} Z^{(1)}_s\right)^p<\infty.
\end{equation}
For each $A > 0$ and each $t \in [0,1]$, set $K^A_t = |H_t| \wedge A$. 
Also, for each $n\in\mn$ and integer $1\leq k\leq n$, set $I_{k,n}:=((k-1)/n, k/n]$ and let $\mathcal{F}_{k,n}$ denote the $\sigma$-algebra generated by $(H_{s-}, Z_s^{(1)})_{0\leq s\leq k/n}$ (we also denote by $\mathcal{F}_{0,n}$ the trivial $\sigma$-algebra). Recalling that $Z^{(1)}$ is a drift-free subordinator we write
\begin{align*}
\E\Big(\int_{(0,\,1]} K_{s-}^A{\rm d}Z^{(1)}_s \Big)^p &= \E\Big(\sum_{k=1}^n \int_{I_{k,n}} K^A_{s-}{\rm d}Z^{(1)}_s\Big)^p\\
  &\leq 2 \E\Big(\sum_{k=1}^n \E\Big( \int_{I_{k,n}} K^A_{s-} {\rm d}Z^{(1)}_s \Big| \mathcal{F}_{k-1,n}\Big) \Big)^p\\
  &\leq 2 \E\Big( \sum_{k=1}^n \int_{I_{k,n}}\E(K^A_s|\mathcal{F}_{k-1,n}) \dd s \Big)^p \Big( \int_{[-1,\,1]} |y| \Lambda_2({\rm d}y) \Big)^p,
\end{align*}
where the first inequality follows by an application of Lemma 6 on p.~411 in \cite{Chow+Teicher:1988}, and the second inequality is a consequence of subadditivity of $x\mapsto x^p$ on $\R_+$ and the equality
$$\E\Big( \int_{I_{k,n}} K^A_{s-} {\rm d}Z^{(1)}_s \Big| \mathcal{F}_{k-1,n}\Big)=\int_{I_{k,n}}\E(K^A_s|\mathcal{F}_{k-1,n}) \dd s \int_{[-1,\,1]} |y| \Lambda_2({\rm d}y)$$ which is implied by the compensation formula for Poisson random measures. Further, letting $n \to \infty$ and using the fact that $(K^A_s)_{s\geq 0}$ is an adapted bounded process, an appeal to Lebesgue's dominated convergence theorem yields
\[
  \lim_{n \to \infty}  \E\Big(\sum_{k=1}^n \int_{I_{k,n}}\E(K^A_s|\mathcal{F}_{k-1, n}) \dd s \Big)^p = \E\Big(\int_0^1 K^A_s \dd s \Big)^p \leq \E(\sup_{s \in [0,\,1]} (K^A_s)^p).
\]
Thus, we have proved that, for each $A>0$,
\begin{align*}
\E\Big(\int_{(0,1]} (|H_{s-}| \wedge A) {\rm d}Z^{(1)}_s \Big)^p &\leq 2 \E(\sup_{s \in [0,1]} |H_s| \wedge A )^p \Big( \int_{[-1,\,1]} |y| \Lambda_2(\dd y) \Big)^p\\ &\leq  \E(\sup_{s \in [0,1]} |H_s|)^p \Big( \int_{[-1,\,1]} |y| \Lambda_2(\dd y) \Big)^p<\infty.
\end{align*}
Letting $A\to\infty$ in the latter formula, we infer \eqref{fin} with the help of Lévy's monotone convergence theorem.
\end{proof}

The result given next is a consequence of Theorem 25.18 in
\cite{Sato:1999}. A direct proof can be found in Lemma 2.1 (a) of
\cite{Aurzada+Iksanov+Meiners:2015}.
\begin{lemma}\label{aim}
Let $p>0$. If $\E e^{-pX_1}<\infty$, then
\[\E\sup_{s\in[0,1]}e^{-pX_s}=\E \exp(-p\inf_{s\in [0,1]}X_s)<\infty.\]
\end{lemma}

\subsection{Proof of Theorem \ref{moments}}

\begin{proof}[Proof of \eqref{mom11}$\Rightarrow$\eqref{mom12}.]
We first
show that conditions \eqref{mom11} ensure $|S|<\infty$ a.s.
Indeed, by H\"{o}lder's inequality $\E e^{-pX_1}<1$ entails $\E
X_1\in (0,\infty]$, whence $\lim_{t\to\infty}X_t=+\infty$ a.s.
Further, $\int_{|y|>1}|y|^p \Lambda_2({\rm d}y)<\infty$ ensures
$\int_{|y|>1}\log |y|\Lambda_2({\rm d}y)<\infty$ and, a fortiori,
the second condition in \eqref{ex}. Now $|S|<\infty$ a.s.\@ follows
from Proposition \ref{maller}.

Now observe that the random variable $S$ can be obtained as a discrete-time perpetuity generated by the pair of random variables
\[
  (M_\ast, Q_\ast):=(e^{-X_{1}}, \int_{[0,\,1]}e^{-X_{s-}}{\rm d}Z_s).\]
In view of the discussion at the beginning of
Section \ref{ma} and our assumption that $\Lambda_2$ is
nontrivial, the distribution of $S$ is nondegenerate. Therefore,
$\P\{Q_\ast+M_\ast r=r\}<1$ for all $r\in\mr$. This enables us to
invoke Proposition \ref{air} which states that $\E|S|^p<\infty$
if, and only if, $\E M_\ast^p=\E e^{-pX_1}<1$ and $\E
|Q_\ast|^p=\E|\int_{[0,\,1]}e^{-X_{s-}}{\rm d}Z_s|^p<\infty$.

It is well-known that the second assumption in \eqref{mom11} is
equivalent to $\E |Z_1|^p<\infty$ (see, for instance, Theorem 25.3
on p.~159 in \cite{Sato:1999}). By Lemma \ref{aim}, the first
condition in \eqref{mom11} guarantees $\E\sup_{s\in
[0,1]}e^{-pX_s}<\infty$. With these at hand we infer
$\E|Q_\ast|^p<\infty$ by Proposition \ref{behme}.
\end{proof}

\begin{proof}[Proof of \eqref{mom12}$\Rightarrow$\eqref{mom11}.]
We assume that $\Lambda_2$ charges all the punctured line
$\R\backslash\{0\}$. Otherwise, the proof becomes simpler. We have
$\E M_\ast^p=\E e^{-pX_1}<1$ by another appeal to Proposition~\ref{air}.
Using the inequality
\[|x+y|^p\geq (2^{1-p}\wedge 1)|x|^p-|y|^p,\quad x,y\in\R\]
which is implied by convexity
(respectively subadditivity) of $s\mapsto s^p$ for $s\geq 0$ when
$p\geq 1$ (resp. when $p\in (0,1)$) we obtain
\begin{align*}
\infty&>\E |S|^p=\E \Big|\int_{[0,\,1]}e^{-X_{s-}}{\rm
d}\tilde Z_s^{(1)}+\int_{[0,\,1]}e^{-X_{s-}}{\rm
d}\tilde Z_s^{(2)}\Big|^p\\&\geq (2^{1-p}\wedge 1)\E
\Big|\int_{[0,\,1]}e^{-X_{s-}}{\rm
d}\tilde Z_s^{(2)}\Big|^p-\E\Big|\int_{[0,\,1]}e^{-X_{s-}}{\rm
d}\tilde Z_s^{(1)}\Big|^p,
\end{align*}
where, for $t \geq 0$,
\begin{align*}
\tilde Z^{(1)}_t &: = \int_{[0,\,t] \times \R} y\1_{[-1,1]}(y) N_2(\dd s \dd
y)=\sum_{\tau_k\leq t} j_k \1_{\{|j_k|\leq 1\}}
\\
\tilde Z^{(2)}_t &: =Z_t-\tilde Z^{(1)}_t=\int_{[0,\,t] \times \R}
y\1_{\R\backslash [-1,1]}(y) N_2(\dd s \dd y)=\sum_{\tau_k\leq
t}j_k\1_{\{|j_k|>1\}}.
\end{align*}
By Theorem 25.3 on p.~159 in
\cite{Sato:1999}, the random variable $|\tilde Z_1^{(1)}|$ has finite
power moments of all positive orders. In particular, $\E
|\tilde Z_1^{(1)}|^p<\infty$. Hence, according to
Proposition~\ref{behme}, $\E\Big|\int_{[0,\,1]}e^{-X_{s-}}{\rm
d}\tilde Z_s^{(1)}\Big|^p<\infty$. Recall the notation $(T_i, R_i)_{i\in\mn}$ introduced in
the proof of Proposition~\ref{behme} for the jump times and jump sizes of $Z^{(2)}$, respectively. Noting that $T_1, T_2,\ldots$ are also the jump times of 
$\tilde Z^{(2)}$, set $V_i:=\tilde Z^{(2)}_{T_i}-\tilde Z^{(2)}_{T_i-}$ for $i\in\mn$ and observe that $|V_i|=R_i$. We infer 
\begin{align*}
\infty>&\E\Big|\int_{[0,\,1]}e^{-X_{s-}}{\rm
d}\tilde Z_s^{(2)}\Big|^p\geq \E\Big|\sum_{T_k\leq
1}e^{-X_{T_k-}}V_k \Big|^p\1_{\{T_1\leq 1<T_2\}}= \E
|e^{-X_{T_1-}}V_1|^p e^{-c}c\\=&\E e^{-pX_{T_1-}}\E|V_1|^p
e^{-c}c,
\end{align*}
where $c=\Lambda_2(\R\backslash [-1,1])$, thereby proving that $\E |V_1|^p<\infty$ or, equivalently, that
the second inequality in \eqref{mom11} holds. The proof of Theorem
\ref{moments} is complete.
\end{proof}

\section{Applications to branching Lévy processes}

\subsection{Definitions and main result}\label{def1}

Branching Lévy processes are a continuous-time generalization
of branching random walks. Similarly to Lévy processes (see
\eqref{levy}), branching Lévy processes are characterized by a
triplet $(\sigma^2,a,\Pi)$, where $\sigma^2 \geq 0$, $a \in \R$
and $\Pi$ is a sigma-finite measure on
\[
  \mathcal{P} := \left\{ \mathbf{x} = (x_n) \in [-\infty,\infty)^\N : x_1 \geq x_2 \geq \cdots \quad \text{and} \quad \lim_{n \to \infty}
  x_n = -\infty \right\}.
\]
Also, it is assumed that $\Pi$ satisfies
\begin{equation}\label{eqn:levyEve}
\int_{\mathcal{P}} (x_1^2 \wedge 1) \Pi(\dd\mathbf{x}) < \infty,
\end{equation}
and that there exists $\theta > 0$ such that
\begin{equation}\label{eqn:finiteExpMoment}
\int_{\mathcal{P}} \left(e^{\theta x_1}\1_{(1,\infty)}(x_1)
+ \sum_{j \geq 2} e^{\theta x_j}\right) \Pi(\dd \mathbf{x}) <
\infty.
\end{equation}
In the sequel we reserve the letter $\theta$ to denote a fixed (possibly
unique) positive number for which \eqref{eqn:finiteExpMoment}
holds.

The set of individuals alive at time $t$ which we denote by
$\mathcal{N}_t$ can be encoded using an adaptation of Ulam-Harris
notation (see \cite{Shi+Watson:2017} for the proposed encoding in
the context of compensated fragmentations). For all $s \leq t$ and
all individual $u$ alive at time $t$, we write $X_s(u)$ for the position at
time $s$ of $u$ if $u \in \mathcal{N}_s$, and for the position of
its ancestor at time $s$ if $u \notin \mathcal{N}_s$.

We outline the evolution of a branching Lévy process with
characteristics $(\sigma^2,a,\Pi)$ and refer to Sections 4 and 5
in \cite{Bertoin+Mallein:2018b} for more details. Denote by $\mathcal{N} = \sum \varepsilon_{(t_k,\x^{(k)})}$ a
Poisson random measure on $\R_+ \times \mathcal{P}$ with mean measure
$\mathrm{LEB}\otimes \Pi$. The position of the initial particle in
the branching Lévy process follows the path of the process
$(X_t(\oslash))_{t\geq 0}$ defined by
\begin{align}\label{ito}
X_t(\oslash):= \sigma B^\ast_t + at &+ \int_{[0,\,t] \times
\mathcal{P}} x_1\1_{[-1,1]}(x_1)\mathcal{N}^c(\dd s \dd
\x)\notag\\&+ \int_{[0,\,t] \times \mathcal{P}}
x_1\1_{\mr\backslash [-1,1]}(x_1)\mathcal{N}(\dd s \dd \x),\quad
t\geq 0,
\end{align}
where $(B^\ast_t)_{t\geq 0}$ is a Brownian motion independent of
$\mathcal{N}$, and the first Poisson integral is taken in the
compensated sense (see Section \ref{defi} for more details concerning a similar integral).
For each atom $(t_k, \x^{(k)})$ of $\mathcal{N}$, the initial
particle gives birth at time $t_k$ to new individuals which are
started at position $X_{t_k-}(\oslash) + x^{(k)}_2,
X_{t_k-}(\oslash) + x^{(k)}_3, \ldots$. Each of the newborn
particles then starts an independent copy of the branching
Lévy process from their birth time and position. Note that
$(X_t(\oslash))_{t\geq 0}$ is a Lévy process with
characteristic triplet $(\sigma^2, a, \Pi_1)$, where $\Pi_1$ is
the image measure of $\Pi$ under the mapping $\x \to x_1$, and
\eqref{ito} is its Lévy-It\^{o} decomposition (compare with
\eqref{levy}). Condition \eqref{eqn:levyEve} guarantees that this
Lévy process is well-defined.

For $z \in \C$ with $\mathrm{Re}(z) = \theta$, set
\[
\kappa(z) = \frac{1}{2} \sigma^2 z^2 + a z + \int_{\mathcal{P}} \left( \sum_{k \geq 1}( e^{z x_k} - 1 - z x_1\1_{(-1,1)}(x_1))
\right) \Pi(\dd \x).
\]
Condition \eqref{eqn:finiteExpMoment} ensures that $\kappa(z)$ is
finite on its domain. By \cite[Theorem
1.1(ii)]{Bertoin+Mallein:2018b}, we have, for $t \geq 0$,
\begin{equation}\label{eqn:cumulant}
\E\left( \sum_{u \in \mathcal{N}_t} e^{z X_t(u)} \right) = \exp(t
\kappa(z)).
\end{equation}
Therefore, it is natural to say that $\kappa(z)$
is the value at $z$ of the {\it cumulant generating function} of
the branching Lévy process.
For later needs we also note that according to the many-to-one
formula for branching Lévy processes (\cite[Lemma
2.2]{Bertoin+Mallein:2018b}), the function $\Psi:\mr\to\mathbb{C}$
defined by
\begin{equation}\label{eqn:defPsi}
\Psi(s):= \kappa(\theta +{\rm i}s) - \kappa(\theta)
\end{equation}
is the Lévy-Khinchine exponent of a Lévy process that we
denote by $\xi=(\xi_t)_{t\geq 0}$.

The branching property of the branching Lévy process tells us
that conditionally on the positions of the particles at time $t$
the processes initiated by these particles are i.i.d.\@ branching
Lévy processes, shifted by the position of their ancestor, see
\cite[Fact (B)]{Bertoin+Mallein:2018}. The branching property in combination with \eqref{eqn:cumulant}
imply that the process $W:=(W_t)_{t\geq 0}$ defined by
\begin{equation}\label{eqn:martingale}
W_t := \sum_{u \in \mathcal{N}_t} e^{\theta X_t(u) - t
\kappa(\theta)},\quad t\geq 0
\end{equation}
is a non-negative continuous-time martingale with respect to the
natural filtration. This martingale, often called Biggins' or
McKean's martingale, and its a.s.\@ limit $W_\infty$ are of
primary importance for the study of branching Lévy processes.
According to a classical result in the field of branching
processes
\[
  \P\{W_\infty = 0\} \in \{ \P\{\exists t > 0 : \mathcal{N}_t = \oslash\}, 1\},
\]
i.e., either $W_\infty$ is strictly positive a.s.\@ on the survival
set of the branching Lévy process or $W_\infty=0$ a.s. While
the first case is equivalent to the uniform integrability of the
martingale $W$, the second one is called the degenerate case.

We are ready to state the second main result of the present
article.
\begin{theorem}\label{thm:martingale}
Let $X$ be a branching Lévy process satisfying
\eqref{eqn:levyEve} and \eqref{eqn:finiteExpMoment}, $W$ the
corresponding Biggins martingale, and $\xi$ the Lévy process
with the Lévy-Khinchine exponent given in \eqref{eqn:defPsi}.
\begin{enumerate}
  \item[(i)] The martingale $W$ is uniformly integrable if, and only
  if,
\begin{align}
  &\nonumber \lim_{t \to \infty}(\theta \xi_t - t \kappa(\theta)) = -\infty~\text{{\rm a.s.}}
  \\
  \text{and} \quad&  \label{eqn:NSCuniformIntegrability}
 \int_{\mathcal{P}} \sum_{k \geq 1} e^{\theta x_k}
\frac{\log \left(\sum_{j \neq k} e^{\theta x_j}\right)}
{A\left(\log\left(\sum_{j \neq k} e^{\theta x_j}\right)\right)}
\1_{(e,\infty)}\Big(\sum_{j\neq k}e^{\theta x_j}\Big) \Pi(\dd
\x)<\infty,
\end{align}
where $A(y) = 1 +\int_{\mathcal{P}} \sum_{k \geq 1} e^{\theta x_k}
\left( (-x_k) \wedge y - 1 \right)_+\Pi(\dd \x)$ for $y\geq 1$.
\item[(ii)] Let $p\in(1,2]$. The martingale $W$ converges in $L_p$ if, and only if,
\begin{equation}\label{eqn:NSClp}
\kappa(p\theta)< p \kappa(\theta) \quad \text{and}\quad
\int_{\mathcal{P}}\sum_{k\geq 1}e^{\theta x_k}\Big(\sum_{j\neq
k}e^{\theta x_j}\Big)^{p-1}
 \1_{(e,\infty)}\Big(\sum_{j\neq k}e^{\theta x_j}\Big)\Pi(\dd
\x)<\infty
\end{equation}
\end{enumerate}
\end{theorem}

In \cite[Theorem 1.1]{Bertoin+Mallein:2018b} similar necessary and
sufficient conditions for the uniform integrability of $W$ were
obtained under the additional assumption that $\E \xi_1\in
(-\infty, \infty)$. A new aspect of part (i) of Theorem
\ref{thm:martingale} is that $\E \xi_1$ may be infinite or not
exist. In \cite[Proposition 1.4]{Bertoin+Mallein:2018b} it was
proved that conditions \eqref{eqn:NSClp} entail the
$L_p$-convergence of $W$ under the additional integrability
condition $\kappa(q\theta) < \infty$ for some $q > p$.

Theorem \ref{thm:martingale} will be proved along the lines of the
proof of the corresponding result for branching random walks, see
the introduction for more details. To this end, in the next
section we define a size-biased measure and the corresponding
spinal decomposition. The latter as well as Proposition
\ref{maller} and Theorem \ref{moments} are essential ingredients
for the proof of Theorem \ref{thm:martingale}.

\subsection{Spinal decomposition}

The spinal decomposition is a useful tool to construct the
branching Lévy process under the size-biased law
\[
  \left.\bar{\P}\right|_{\calF_t} :=\left.W_t\P\right|_{\calF_t},\quad t \geq 0,
\]
where $(\mathcal{F}_t)_{t\geq 0}$ is the natural filtration for
$W$. The resulting process is a branching process with the set of
distinguished individuals, called the spine. While the individuals
belonging to the spine produce offspring and displace them
according to a special law, the rest of the population behaves as
in a standard branching Lévy process. This justifies the term
`spinal decomposition'.

To explain the evolution of a branching Lévy
process with spine we need more notation.  Let $\hat{\Pi}$ be a
measure on $\mathcal{P}\times \N$ defined by
\begin{equation}\label{eqn:hatPi}
\hat{\Pi}(\dd \x \dd k) = e^{\theta x_k}\left(\Pi(\dd x)
\mathrm{Count}(\dd k)\right),
\end{equation}
where $\mathrm{Count}$ is the counting measure on $\N$. Set
\[\hat{a} = a + \theta \sigma^2 + \int_{\mathcal{P}}\Big(\sum_{k \geq 1} x_k e^{\theta x_k}\1_{[-1,1]}(x_k) - x_1\1_{[-1,1]}(x_1)\Big)\Pi(\dd \x)
\]
and note that $\hat a$ is well-defined and finite by
\eqref{eqn:levyEve} and \eqref{eqn:finiteExpMoment}. Also, we
denote by $\hat{\mathcal{N}}$ a Poisson random measure on $\R_+
\times \mathcal{P} \times \N$ with mean measure $\mathrm{LEB}
\otimes \hat{\Pi}$ and by $(\hat{B}_t)_{t\geq 0}$ a Brownian
motion which is independent of $\hat{\mathcal{N}}$.

Now we define the spine process
$\hat{\xi}=(\hat{\xi}_t)_{t\geq 0}$ by the following
Lévy-It\^{o} decomposition: for $t\geq 0$
\begin{align*}
\hat{\xi}_t := &\sigma \hat{B}_t + \hat{a} t +
\int_{[0,\,t]\times \mathcal{P}\times \N} x_k
\1_{[-1,1]}(x_k)\hat{\mathcal{N}}^{(c)}(\dd s\dd \x\dd k)\\&+
\int_{[0,\,t]\times \mathcal{P}\times \N} x_k\1_{\R\backslash
[-1,1]}(x_k) \hat{\mathcal{N}}(\dd s\dd \x\dd k).
\end{align*}
Plainly, $\hat{\xi}$ is a Lévy process with characteristic
triplet $(\sigma^2, \hat{a}, \Lambda_1)$, where
the Lévy measure is given by
\begin{equation}\label{la1}
\int_\mr f(-x) \Lambda_1(\dd x) = \int_{\mathcal{P}}\Big(\sum_{k
\geq 1} e^{\theta x_k} f(x_k)\Big) \Pi(\dd \x).
\end{equation}
Further, it can be checked that the Lévy-Khinchine exponent of
$\hat{\xi}$ is $\Psi$ defined in \eqref{eqn:defPsi}.

We are now ready to discuss briefly the evolution
of a branching Lévy process with spine. The spine particle
displaces according to the Lévy process $\hat{\xi}$, and for
each atom $(t,\x,k)$ of $\hat{\mathcal{N}}$, the spine particle
produces offspring at positions $\hat{\xi}_{t-} + x_j$ for all
$j\neq k$. Each of these newborn particles then immediately starts
an independent branching Lévy process from their birth place
and time. Retaining the notation $\mathcal{N}_t$ and $X_s(u)$ (see
Section \ref{def1}) for the branching Lévy process with spine
we shall also write $w_t$ for the label at time $t$ of the spine
particle. With these at hand we denote by $\hat{\P}$ the law of
$((X_t(u))_{u\in\mathcal{N}_t, t\geq 0}, (\mathcal{N}_t)_{t\geq
0}, (w_t)_{t\geq 0})$.

Denote by $(\mathcal{H}_t)_{t\geq 0}$ the filtration associated to
$(X_t(u))_{u \in \mathcal{N}_t, t \geq 0}$ for the branching
Lévy process with spine which excludes the information
concerning the labels of the spine individuals.

\begin{lemma}
We have $\bar{\P}_{|\mathcal{H}_t} = \hat{\P}_{|\mathcal{H}_t}$
for $t\geq 0$ and
\[
  \hat{\P}\{w_t=u|\mathcal{H}_t\} = \frac{e^{\theta X_t(u) - t
  \kappa(\theta)}}{W_t},\quad t\geq 0.
\]
Furthermore, under $\hat{\P}$, $(X_t(w_t))_{t\geq 0}$ is a
Lévy process with Lévy-Khinchine exponent $\Psi$.
\end{lemma}

The spinal decomposition was introduced in
\cite{Lyons+Pemantle+Peres:1995} in the context of Galton-Watson
processes. Lyons \cite{Lyons:1997} then proved a spinal
decomposition result for branching random walks. This result was
further generalized to branching Markov chains and general
associated harmonic functions in \cite{Biggins+Kyprianou:2004}, to
general Markov processes and multiple spines in
\cite{Harris+Roberts:14}, etc.
In the context of growth-fragmentation processes a proof of the
spinal decomposition appeared in
\cite{Bertoin+Budd+Curien+Kortchemski:2016} for binary compensated
fragmentations, i.e., under the assumption $\Pi(\{x_1
> 0\}) +\Pi(\{x_3>-\infty\}) = 0$. The first general spinal
decomposition result for branching Lévy processes was obtained
in \cite[Theorem 5.2]{Shi+Watson:2017} under the assumption
$\Pi(\{ x_1 > 0\}) = 0$. A simple argument was given in
\cite[Lemma 2.3]{Bertoin+Mallein:2018b} which enabled one to
deduce the spinal decomposition for branching Lévy processes
from that for branching random walks.

\subsection{Proof of Theorem \ref{thm:martingale}}

We start with some preliminary work. Denote by $\Omega_s$ the
multiset\footnote{I.e., the set of elements counted with their
multiplicity.} of children's positions at time $s$ relative to the
positions of their parents belonging to the spine, i.e.,
\[
  \Omega_s =
  \begin{cases}
    \oslash, & \text{if } \hat{\mathcal{N}}(\{s\} \times \mathcal{P} \times \N) = 0\\
    \{(x_j)_{j\neq k}\}, & \text{if } \hat{\mathcal{N}}(\{(s,\x,k\}) = 1.
  \end{cases}
\]
Setting
\[
S_t := \sum_{0 \leq s \leq t} e^{\theta
X_{s-}(w_{s-})-s\kappa(\theta)} \sum_{z \in \Omega_s} e^{\theta
z}, \quad t\geq 0
\]
we note that the $\hat{\P}$-\text{a.s.} limit
$\lim_{t\to\infty} S_t$, provided it is finite, is a Lévy-type
perpetuity (see \eqref{perp}) in which the role of $X$ is played
by $(-\theta X_t(w_t)+t\kappa(\theta))_{t\geq 0}$ under
$\hat{\P}$, and the associated Lévy measures $\Lambda_1$
and~$\Lambda_2$ are given, respectively, by
\eqref{la1} and
\begin{equation*}
\int_{\mr_+} f(x) \Lambda_2(\dd x) = \int_{\mathcal{P}} \sum_{k
\geq 1} e^{\theta x_k} f\left( \sum_{j \neq k} e^{\theta x_j}
\right) \Pi(\dd \x).
\end{equation*}
It can be checked that assumptions
\eqref{eqn:levyEve} and \eqref{eqn:finiteExpMoment} guarantee that
the so defined $\Lambda_1$ and $\Lambda_2$ satisfy \eqref{assump}.

To facilitate a forthcoming application of Proposition
\ref{maller} let us note that the second condition in
\eqref{eqn:NSCuniformIntegrability} is equivalent to
\begin{equation}\label{inter}
\int_{(e,\infty)} \frac{\log y}{A(\log y)} \Lambda_2(\dd y) <
\infty,
\end{equation}
where $A(x) = 1 + \int_1^x \Lambda_1((y,\infty)) \dd y$ for $x\geq
1$ as in Section \ref{defi} but with $\Lambda_1$ as defined above.
As far as an application of Theorem \ref{moments} is concerned
observe that $\kappa(p\theta)< p \kappa(\theta)$ which is the
first condition in \eqref{eqn:NSClp} is equivalent to
\begin{equation}\label{inter200}
\hat{\E}\exp((p-1)(\theta
X_t(w_t)-t\kappa(\theta)))=\exp(\kappa(p\theta)-p\kappa(\theta))<1.
\end{equation}
The latter is the first condition in \eqref{mom11} with $X$ as
defined in the previous paragraph. Further, the second condition
in \eqref{eqn:NSClp} is equivalent to
\begin{equation}\label{inter100}
\int_{(1,\infty)}y^{p-1}\Lambda_2({\rm d}y)<\infty.
\end{equation}
Now we write a basic representation for what follows:
\begin{equation}\label{eqn:conditionalExpectation}
W^*_t := \hat{\E}\left( W_t \middle| \mathcal{G}\right) =
e^{\theta X_t(w_t) - t \kappa(\theta)} + S_t,\quad t\geq 0,
\end{equation}
where $\mathcal{G}$ is the $\sigma$-algebra which contains the
information concerning the trajectory of the spine as well as the
birth place and the birth times of its offspring.

Passing to the proof of Theorem \ref{thm:martingale} we first deal
with the uniform integrability of $W$.
\begin{lemma}\label{lem:unifIntegrability} Under the assumptions of Theorem
\ref{thm:martingale} the martingale $W$ is uniformly integrable
if, and only if, conditions \eqref{eqn:NSCuniformIntegrability}
hold.
\end{lemma}
\begin{proof}
We use the classical observation (see, for instance, p.~220 in
\cite{Lyons:1997}) that
\begin{equation}\label{equi}
W ~\text{is uniformly integrable under } \P \iff \bar{W}_\infty
:=\limsup_{t \to \infty} W_t <\infty \quad \hat{\P}-\text{a.s.}
\end{equation}
Therefore, it is enough to prove that conditions
\eqref{eqn:NSCuniformIntegrability} are equivalent to the
$\hat{\P}$-a.s.\@ finiteness of $\bar{W}_\infty$.

Assume that conditions \eqref{eqn:NSCuniformIntegrability} hold.
Since the law of the Lévy process $(\xi_t)_{t\geq 0}$ is the
same as the $\hat{\P}$-law of $(X_t(w_t))_{t\geq 0}$, the first
condition in \eqref{eqn:NSCuniformIntegrability} ensures that
\begin{equation}\label{inter3}
\lim_{t\to\infty}(\theta X_t(w_t) - t \kappa(\theta))=
-\infty\quad \hat{\P}-\text{a.s.}
\end{equation}
This entails
$\lim_{t \to \infty} W^*_t = \lim_{t \to \infty} S_t$ $\hat{\P}$-a.s. With \eqref{inter} and \eqref{inter3} at hand, an application of
Proposition \ref{maller} (recall our specific choice of $X$)
yields $\lim_{t \to \infty} S_t<\infty$ $\hat{\P}$-a.s.\@ and thereupon
$\lim_{t \to \infty} W^*_t<\infty$ $\hat{\P}$-a.s. Invoking the
conditional Fatou lemma we further infer
\begin{equation}\label{inter2}
{\lim\inf}_{t\to\infty}W_t<\infty\quad \hat{\P}-\text{a.s.}
\end{equation}
According to Proposition 2 in \cite{Harris+Roberts:2009}, $1/W$ is
a positive supermartingale under $\hat{\P}$. Thus, $1/W_t$
converges $\hat{\P}$-a.s.\@ as $t\to\infty$. In view of
\eqref{inter2} the limit cannot be zero.
Therefore, $\bar{W}_\infty<\infty$
$\hat{\P}$-a.s.\@ which is equivalent to the uniform integrability
of $W$.

Conversely, assume that $W$ is uniformly integrable or
equivalently $\bar{W}_\infty<\infty$ $\hat{\P}$-a.s. Then
\[
  W_t\geq \sum_{u \in \mathcal{N}_t} e^{\theta X_t(u) - t \kappa(\theta)} \geq e^{\theta X_t(w_t) - t
  \kappa(\theta)},\quad t\geq 0
\]
entails $\limsup_{t \to \infty}(\theta X_t(w_t) - t
\kappa(\theta)) < \infty$ $\hat{\P}$-a.s, whence
$\lim_{t \to \infty}(\theta X_t(w_t) - t \kappa(\theta))=-\infty$
$\hat{\P}$-a.s. This proves that the first condition in
\eqref{eqn:NSCuniformIntegrability} holds.

Passing to the proof of the second condition in
\eqref{eqn:NSCuniformIntegrability} we first observe that, for all
$0 \leq s \leq t$,
\[
  W_t \geq \sum_{0 \leq r \leq s} e^{\theta X_{r-}(w_r) - r \kappa(\theta)} \sum_{z \in \Omega_r} e^{\theta z} W_t^{(r,z)}\quad
\hat{\P}-\text{a.s.},
\]
where the random variables $W_t^{(r,z)}:= \sum_{u
\in \mathcal{N}_t} e^{\theta (X_t(u) - X_r(u)) - (t-r)
\kappa(\theta)} \ind{u \text{ descendant of } z}$
are independent of $\mathcal{G}$ and have the same
$\hat{\P}$-distribution as the $\P$-distribution of $W_{t-r}$.
Letting now $t\to\infty$ we infer, for all $s\geq 0$,
\begin{equation}\label{eqn:boundedness}
\bar{W}_\infty \geq \sum_{r \leq s}  e^{\theta X_{r-}(w_r) - r
\kappa(\theta)} \sum_{z \in \Omega_r} e^{\theta z}
W^{(r,z)}_\infty\quad \hat{\P}-\text{a.s.},
\end{equation}
where $W_\infty^{(r,z)}$ is the limit of the
Biggins martingale associated to the descendant of the spine born
at time $r$ at position $z$.

The random variables $(W_\infty^{(r,z)})_{r\geq 0,
z \in \Omega_r}$ are i.i.d. In view of the assumption
$\bar{W}_\infty<\infty$ $\hat{\P}$-a.s.\ equivalence \eqref{equi}
ensures $\E W^{(r,z)}_\infty= 1$. As a consequence, there exists
$\delta > 0$ such that $\P\{W^{(r,z)}_\infty \geq 1\} = \delta$.
Setting $e^{(r,z)} = \1_{[1,\infty)}(W^{(r,z)}_\infty)$ we
conclude that the random variables $(e^{(r,z)})_{r\geq 0, z \in
\Omega_r}$ are independent Bernoulli random variables with
parameter $\delta$. Now \eqref{eqn:boundedness} implies that, for
all $s\geq 0$,
\[
  \bar{W}_\infty \geq \sum_{r \leq s}  e^{\theta X_{r-}(w_r) - r \kappa(\theta)} \sum_{z \in \Omega_r} e^{\theta z} e^{(r,z)}=:\Gamma_s\quad \hat{\P}-\text{a.s.}
\]
In particular, there exists a sequence $(s_j)$
such that $\lim_{j\to\infty}\Gamma_{s_j}<\infty$ $\hat{\P}$-a.s.

Assume now that $\lim_{t\to\infty} S_t=\infty$
$\hat{\P}$-a.s., so that
$\lim_{j\to\infty}(\Gamma_{s_j}/S_{s_j})=0$ $\hat{\P}$-a.s. Since
$\Gamma_{s_j}/S_{s_j}\leq 1$ $\hat{\P}$-a.s.\,
$\Gamma_{s_j}/S_{s_j}$ must converge to $0$ in $\hat{\P}$- mean.
However, this is not the case, for
$\hat{\E}(\Gamma_{s_j}/S_{s_j})=\delta$, a contradiction. Thus, we
have shown that $\lim_{t\to\infty} S_t<\infty$ $\hat{\P}$-a.s. By
Proposition \ref{maller} this implies that the second condition in
\eqref{eqn:NSCuniformIntegrability} holds. The proof of Lemma~\ref{lem:unifIntegrability} is complete.
\end{proof}

The proof of the second part of Theorem \ref{thm:martingale}
follows by a similar reasoning. We first use the fact that
$(W_n)_{n\in\mn_0}$ is the Biggins martingale of a branching
random walk with the underlying point process $\sum_{u \in
\mathcal{N}_1} \varepsilon_{X_1(u)}$. The following result is
well-known and can be found in Theorem 3.1 of
\cite{Alsmeyer+Kuhlbusch:2010}, Corollary 5 of
\cite{Iksanov:2004}, Theorem 2.1 of \cite{Liu:2000} and
perhaps some other articles:
the $L_p$-convergence of $(W_n)_{n\in\mn_0}$ for $p>1$ is
equivalent to the following two conditions
\begin{equation}\label{eqn:NSCBiggins}
\kappa(p\theta)< p \kappa(\theta) \quad \text{and} \quad \E W_1^p
< \infty.
\end{equation}
Another form of the left-hand inequality is given by the first
inequality in
\[1>\E\sum_{u\in\mathcal{N}_1}e^{p(\theta
X_1(u)-\kappa(\theta))}=e^{\kappa(p\theta)-p\kappa(\theta)}.\]
As the $L_p$-convergence of $W$ is obviously equivalent to that of
$(W_n)_{n\in\mn_0}$, it only remains to check that conditions
\eqref{eqn:NSClp} and \eqref{eqn:NSCBiggins} are equivalent.
\begin{lemma}\label{aux2}
Let $p\in (1,2]$. Assume \eqref{eqn:levyEve} and
\eqref{eqn:finiteExpMoment} hold and that $\kappa(p\theta) < p
\kappa(\theta)$. Then
\[
  \E W_1^p < \infty \iff \int_{\mathcal{P}}\sum_{k\geq 1}e^{\theta x_k}\Big(\sum_{j\neq
k}e^{\theta x_j}\Big)^{p-1}
 \1_{(e,\infty)}\Big(\sum_{j\neq k}e^{\theta x_j}\Big)\Pi(\dd
\x)<\infty.
\]
\end{lemma}

\begin{proof}
$\Leftarrow$: We intend to prove that $\E W_1^p < \infty$ or
equivalently $\hat{\E} W_1^{p-1} < \infty$. By Theorem~\ref{moments},
conditions \eqref{inter200} and \eqref{inter100}
ensure that $S:=\lim_{t\to\infty} S_t<\infty$
$\hat{\P}$-\text{a.s.} and that $\hat{\E}S^{p-1}<\infty$.
Recalling \eqref{eqn:conditionalExpectation} we obtain
\[\hat{\E} W_1^{p-1}\leq
\hat{\E}[\hat{\E}(W_1|\mathcal{G})^{p-1}]\leq
\hat{\E}\big(e^{(p-1)(\theta
X_1(w_1)-\kappa(\theta))}+S_1^{p-1}\big)< \infty\]
having used the conditional Jensen inequality for the first inequality,
subadditivity of $x\mapsto x^{p-1}$ on $\R_+$ for the second, and
\eqref{inter200} together with $\hat{\E}S_1^{p-1}\leq
\hat{\E}S^{p-1}$ for the third.

 $\Rightarrow$: For $s>0$ and $z\in \Omega_s$, denote by
$(W^{(s,z)}_u)_{u\geq 0}$ the Biggins martingale associated to the
descendant of the spine born at time $s$ at position $z$. Setting
$\underline{W}^{(s,z)}_1: = \inf_{u \in [0,1]} W^{(s,z)}_u$ we
obtain
\[
W_1 \geq \sum_{0 \leq s \leq 1} e^{\theta X_{s-}(w_{s-}) - s
\kappa(\theta)} \sum_{z \in \Omega_s} e^{\theta z} W^{(s,z)}_{1-s}
\geq \sum_{0 \leq s \leq 1} e^{\theta X_{s-}(w_{s-}) - s
\kappa(\theta)} \sum_{z \in \Omega_s} e^{\theta
z}\underline{W}^{(s,z)}_1\quad \hat{\P}-\text{a.s.}
\]
The random variables $\underline{W}^{(s,z)}_1$ are
$\hat{\P}$-i.i.d., positive with positive probability and
independent of all the other random variables occurring under the
sum. Using concavity of $x\mapsto x^{p-1}$ on $\R_+$ yields
\[
  W_1^{p-1} \geq S_1^{p-1} \times \frac{\sum_{0 \leq s \leq 1, z \in \Omega_s} e^{\theta X_{s-}(w_{s-}) - s
\kappa(\theta)} e^{\theta z}
(\underline{W}^{(s,z)}_1)^{p-1}}{S_1}\quad
\hat{\P}-\text{a.s.}
\]

Denoting by $\underline{W}_1$ a generic copy of $\underline{W}_1^{(s,z)}$, we deduce that
$
\hat{\E} W_1^{p-1}\geq \hat{\E}S_1^{p-1}
\hat{\E}\underline{W}_1^{p-1},
$
thereby showing that $\hat{\E}S_1^{p-1}<\infty$.

Using Proposition
\ref{air} in the same way as in the proof of Theorem
\ref{moments}, implication \eqref{mom11}$\Rightarrow$
\eqref{mom12} we conclude that $\hat{\E}S_1^{p-1}<\infty$ together
with \eqref{inter200} ensure that $\hat{\E}S^{p-1}<\infty$.
Formula \eqref{inter100} then follows by Theorem \ref{moments}.
The proof of Lemma \ref{aux2} is complete.
\end{proof}

\paragraph*{Acknowledgements.} The authors thank the Associate Editor for having detected a serious error in the earlier version of this paper.


\begin{thebibliography}{99}

\small

\bibitem{Alsmeyer+Buraczewski+Iksanov:2017} G. Alsmeyer, D. Buraczewski and A. Iksanov, \textit{Null recurrence
and transience of random difference equations in the contractive
case}. J. Appl. Probab. \textbf{54} (2017), 1089--1110.

\bibitem{Alsmeyer+Iksanov:2009} G. Alsmeyer and A. Iksanov, \textit{A log-type moment result for perpetuities and its
application to martingales in supercritical branching random
walks}. Electron. J. Probab. \textbf{14} (2009), 289--313.

\bibitem{Alsmeyer+Iksanov+Roesler:2009} G. Alsmeyer, A. Iksanov and U. R\"{o}sler, \textit{On distributional properties of
perpetuities.} J. Theoret. Probab. \textbf{22} (2009), 666--682.

\bibitem{Alsmeyer+Kuhlbusch:2010} G. Alsmeyer and D. Kuhlbusch,
\textit{Double martingale structure and existence of
$\phi$-moments for weighted branching processes}. M\"{u}nster J.
Math. \textbf{3} (2010), 163--212.

\bibitem{Aurzada+Iksanov+Meiners:2015} F. Aurzada, A. Iksanov and M. Meiners, \textit{Exponential moments of first passage times and related quantities
for L\'{e}vy processes}. Math. Nachr. \textbf{288} (2015),
1921--1938.

\bibitem{Behme:2011} A.~D. Behme, \textit{Distributional properties of solutions of ${\rm d}V_t=V_{t-}{\rm d}U_t+{\rm d}L_t$ with L\'{e}vy noise}.
Adv. Appl. Probab. \textbf{43} (2011), 688--711.

\bibitem{Bertoin+Budd+Curien+Kortchemski:2016} J. Bertoin, T. Budd, N. Curien and I. Kortchemski,
\textit{Martingales in self-similar growth-fragmentations and their connections with random planar maps}, to appear in Probab. Theory Related Fields (2018).

\bibitem{Bertoin+Lindner+Maller:2008} J. Bertoin, A. Lindner and R. Maller, \textit{On continuity
properties of the law of integrals of L\'evy processes}.
S\'{e}minaire de Probabilit\'{e}s XLI, Lecture Notes in
Mathematics \textbf{1934} (2008), 137--159.

\bibitem{Bertoin+Mallein:2018} J. Bertoin and B. Mallein,
\textit{Infinitely ramified point measures and branching {L}\'evy processes}, to appear in Ann. Probab. (2018).

\bibitem{Bertoin+Mallein:2018b} J. Bertoin and B. Mallein,
\textit{Biggins' martingale convergence for branching L\'{e}vy
processes}. Electron. Commun. Probab. \textbf{23} (2018), paper
no. 83, 12 pp.

\bibitem{Biggins:1977} J.~D. Biggins,
\textit{Martingale convergence in the branching random walk}. J.
Appl. Probab. \textbf{14} (1977), 25--37.

\bibitem{Biggins:1992} J.~D. Biggins,
\textit{Uniform convergence of martingales in the branching random
walk}. Ann. Probab. \textbf{20} (1992), no. 1, 137-151.

\bibitem{Biggins+Kyprianou:2004} J.~D. Biggins and A.~E. Kyprianou,
\textit{Measure change in multitype branching}. Adv. Appl. Probab.
\textbf{36} (2004), 544--581.

\bibitem{Buraczewski+Damek+Mikosch:2016} D. Buraczewski, E. Damek and T. Mikosch,
\textit{Stochastic models with power-law tails: the equation $X =
AX + B$}. Springer, 2016.

\bibitem{Buraczewski+Dyszewski+Iksanov+Marynych:2018} D. Buraczewski, P. Dyszewski, A. Iksanov and A.
Marynych, \textit{On perpetuities with gamma-like tails}. J. Appl.
Probab. \textbf{55} (2018), 386--389.

\bibitem{Chow+Teicher:1988} Y.~S. Chow and H. Teicher, \textit{Probability theory:
independence, interchangeability, martingales}. Springer, 1988.

\bibitem{Damek+Kolodziejek:2018} E. Damek and B. Ko{\l}odziejek,
\textit{A renewal theorem and supremum of a perturbed random
walk}. Electron. Commun. Probab. \textbf{23} (2018), paper no. 82,
13 pp.

\bibitem{Erickson+Maller:2005} K.~B. Erickson and R.~A. Maller,
\textit{Generalized Ornstein-Uhlenbeck processes and the
convergence of L\'{e}vy integrals}. In: S\'{e}minaire de
Probabilit\'{e}s XXXVIII (M. Emery, M. Ledoux, M. Yor, Eds.).
Lect. Notes Math. \textbf{1857}, Springer, Berlin. (2005), 70--94.

\bibitem{Goldie+Maller:2000} C.~M. Goldie and R.~A. Maller, \textit{Stability of perpetuities}. Ann. Probab. \textbf{28}
(2000), 1195--1218.

\bibitem{Harris+Roberts:2009} S.~C. Harris and M.~I. Roberts, \textit{Measure changes with extinction}. Statist. Probab. Letters. \textbf{79}
(2009), 1129--1133.

\bibitem{Harris+Roberts:14} S. Harris and M. Roberts, \textit{The many-to-few lemma and multiple spines}.
Ann. Inst. H. Poincar\'{e} Probab. Statist. \textbf{53} (2017),
226--242.

\bibitem{Iksanov:2004} A.~M. Iksanov, \textit{Elementary fixed points of the BRW
smoothing transforms with infinite number of summands}. Stoch.\@
Proc.\@ Appl. \textbf{114} (2004), 27--50.

\bibitem{Iksanov:2016} A. Iksanov, \textit{Renewal theory for perturbed random walks and similar processes}. Birkh\"{a}user, 2016.

\bibitem{Iksanov+Jedidi+Bouzzefour:2018} A. Iksanov, W. Jedidi and F. Bouzzefour, \textit{Functional limit theorems for the number of busy servers in a
$G/G/\infty$ queue}. J. Appl. Probab. \textbf{55} (2018), 15--29.

\bibitem{Iksanov+Pilipenko+Samoilenko:2017} A. Iksanov, A. Pilipenko and I. Samoilenko, \textit{Functional limit theorems for the maxima of perturbed
random walks and divergent perpetuities in the $M_1$-topology}.
Extremes. \textbf{20} (2017), 567--583.

\bibitem{Liu:2000} Q. Liu, \textit{On generalized multiplicative cascades}. Stoch. Proc. Appl. \textbf{86} (2000), 263--286.

\bibitem{Lyons+Pemantle+Peres:1995} R. Lyons, R. Pemantle, Y. Peres, \textit{Conceptual Proofs of $L
\log L$ criteria for mean behavior of branching processes}. Ann.
Probab. \textbf{23} (1995),
1125--1138.

\bibitem{Lyons:1997} R. Lyons, \textit{A simple path to Biggins' martingale
convergence for branching random walk}. Classical and modern
branching processes, IMA Volumes in Mathematics and its
Applications. \textbf{84}, 217--221, Springer, 1997.

\bibitem{Sato:1999} K.-I. Sato, \textit{L\'{e}vy processes and infinitely divisible
distributions}. Cambridge University Press, 1999.

\bibitem{Shi+Watson:2017} Q. Shi and A. Watson,
\textit{Probability tilting of compensated fragmentations}. Preprint available at {\tt
https://arxiv.org/abs/1707.00732}
\end{thebibliography}
\end{document}